\documentclass{birkjour}
\usepackage{subfigure}
 \newtheorem{thm}{Theorem}[section]
 \newtheorem{corollary}[thm]{Corollary}
 \newtheorem{lemma}[thm]{Lemma}
 \newtheorem{Proposition}[thm]{Proposition}
 \theoremstyle{definition}
 
 \theoremstyle{remark}
 
 \newtheorem{example}{Example}
 \numberwithin{equation}{section}
 
 \newcommand{\R}{\mathbb{R}}
 \newcommand{\x}{\mathbf{x}}
  
  \renewcommand{\a}{\mathbf{a}}
  
   \newcommand{\w}{\mathbf{w}}
  \renewcommand{\c}{\mathbf{c}}
  \newcommand{\n}{\mathbf{n}}
   \newcommand{\q}{\mathbf{q}}
    \newcommand{\F}{\mathbf{F}}

 \newcommand{\fsep}{\hspace*{\fill}}

\begin{document}

%
%
%
%
%
%
%
%
%

\title[Improper Indefinite Affine Spheres]
 {A Geometric Representation of Improper \\ Indefinite Affine Spheres with Singularities}

\author[M.Craizer]{Marcos Craizer}

\address{%
Departamento de Matem\'{a}tica- PUC-Rio\br
Rio de Janeiro\br
Brazil}

\email{craizer@puc-rio.br}

\thanks{The first and second authors want to thank CNPq for financial support during the preparation of this manuscript.}
\author[R.C.Teixeira]{Ralph C. Teixeira}
\address{Departamento de Matem\'{a}tica Aplicada- UFF \br
Niter\'oi\br
Brazil}
\email{ralph@mat.uff.br}

\author[M.A.H.B.daSilva]{Moacyr A.H.B. da Silva}
\address{Centro de Matem\'{a}tica Aplicada- FGV \br
Rio de Janeiro\br
Brazil}
\email{moacyr@fgv.br}


\date{July 13, 2010}

\begin{abstract}
Given a pair of planar curves, one can define its generalized area distance,  a concept that generalizes the area distance of a single curve. 
In this paper, we show that the generalized area distance of a pair of planar curves is an improper indefinite affine spheres with singularities, and, reciprocally, 
every indefinite improper affine sphere in $\R^3$ is the generalized distance of a pair of planar curves. 
Considering this representation, the singularity set of the improper affine sphere corresponds to
the area evolute of the pair of curves, and this fact allows us to describe a clear geometric picture of the former.
Other symmetry sets of the pair of curves, like the affine area symmetry set and the affine envelope symmetry set can 
be also used to describe geometric properties of the improper affine sphere.
\end{abstract}

\maketitle

\section{Introduction}

The area based distance to a convex planar curve at a point $p$ is defined as the minimum area of a region bounded by the curve and a line through $p$  
and is a well known concept in computer vision  (\cite{Niet04}).  In this paper we consider a generalization of this distance to a pair of planar curves as follows: 
for a pair of parameterized planar curves $(\alpha(s),\beta(t))$, $s\in I_1$, $t\in I_2$, denote by $\x(s,t)$ the midpoint of the chord $\overline{\alpha(s)\beta(t)}$ 
and by $g(s,t)$ the area of the region bounded by the two curves, the 
chord $\overline{\alpha(s)\beta(t)}$ and some other arbitrary fixed chord. We shall call the map $(s,t)\to(\x(s,t),g(s,t))$ the {\it generalized area distance} of the pair of curves 
$(\alpha(s),\beta(t))$. When $\alpha$ and $\beta$ are parameterizations of the same convex curve without parallel tangents, the function $g$ is exactly the area distance of the curve (\cite{Craizer08}).

It turns out that any generalized area distance map is an improper indefinite affine sphere with singularities. Reciprocally, any improper 
indefinite affine sphere is obtained as a generalized distance of a pair of planar curves. This correspondence between generalized area distance 
and improper indefinite affine spheres is the main theme of this paper.

The singularity set of the generalized area distance is formed by the pairs $(s,t)$ whose corresponding velocity vectors $\alpha'(s)$ and $\beta'(t)$ are parallel.  
Thus it coincides with a well-known symmetry set associated with the curves $\alpha$ and $\beta$,
the {\it area evolute}, also called {\it midpoint parallel tangent locus} (\cite{Giblin08},\cite{Holtom00}).

Improper affine maps were introduced in \cite{Martinez05} for convex surfaces as a class of improper affine spheres with singularities. For non-convex surfaces,  a similar class of improper affine spheres with singularities was introduced in \cite{Nakajo09}, where they were called {\it improper indefinite affine maps}, shortly 
{\it indefinite IA-maps}. In our representation, elements of this class correspond to generalized area distances such that $\alpha$ and $\beta$ are regular curves, i.e., $\alpha'(s)\neq 0$, $s\in I_1$ and $\beta'(0)$, $s\in I_2$. 
The representation of \cite{Nakajo09} is based on two para-holomorphic functions, and it is 
easy to pass from them to the pair or planar curves of our representation and vice-versa. The advantage of working with our representation is that 
we gain a lot of geometric insight about the indefinite improper affine maps.
For example, the nature of the non-degenerate singular points of the indefinite IA-maps is studied in \cite{Nakajo09}, and we do a similar study here, but correlating with the geometry of the 
area evolute of the pair of planar curves. 



Besides the area evolute, there are some other well known symmetry sets associated with a pair of plane curves, namely,  the affine envelope symmetry set (AESS) 
and the affine area symmetry set (AASS) (\cite{Giblin98},\cite{Giblin08},\cite{Holtom00}). Each of these sets
can be transported to the surface by our representation formula. The AASS corresponds 
to the self-intersection set of the affine sphere, while the AESS corresponds to points of  local symmetry of the affine sphere.

The paper is organized as follows: in section 2 we review some well-known facts about affine immersions. In section 3, we describe the promised
representation formula. In section 4, we study the singularity set of the improper affine sphere while in section 5 we discuss the subsets of the affine sphere that correspond to the AESS and the AASS.


\section{ Asymptotic coordinates and improper affine spheres}

Consider an immersion $\q:I_1\times I_2\to\R^3$, where $I_1\subset\R$ and $I_2\subset\R$ are intervals. For $s\in I_1$, $t\in I_2$, we write
\begin{eqnarray*}
L&=&\left[\q_s,\q_t,\q_{ss}\right]\\
M&=&\left[\q_s,\q_t,\q_{st}\right]\\
N&=&\left[\q_s,\q_t,\q_{tt}\right],
\end{eqnarray*}
where subscripts denotes partial derivatives and $[\cdot,\cdot,\cdot]$ denotes determinant. We shall use the same notation $[\cdot,\cdot]$ for the determinant of two planar vectors. 
The surface $S=\q(I_1\times I_2)$ is convex if and only if $LN-M^2>0$. Along this paper we shall assume that $LN-M^2<0$, and 
thus we shall consider only non-convex surfaces. 
In case $L=N=0$, we say that $(s,t)$ are asymptotic coordinates of the surface $S$.

Assuming that $\q$ is an asymptotic parameterization, we may also assume $M>0$ by changing the roles of $s$ and $t$ if necessary. Writing $\Omega=\sqrt{M}$, the affine normal $\xi$ is given by
\begin{equation}\label{AffineNormal}
\xi(s,t)=\frac{\q_{st}}{\Omega}.
\end{equation}
We also write 
\begin{eqnarray*}
\q_{ss}&=&\frac{\Omega_s}{\Omega}\q_s+\frac{A}{\Omega}\q_t\\
\q_{tt}&=&\frac{D}{\Omega}\q_s+\frac{\Omega_t}{\Omega}\q_t,
\end{eqnarray*}
where $A$ and $D$ are coefficientes of the cubic form $C(s,t)$, i.e.,
\begin{equation*}
C(s,t)=Ads^3+Ddt^3
\end{equation*}
(for details see \cite{Buchin83}). 

An immersion $\q$ as above is called an improper affine sphere if the affine normal $\xi$ is a constant vector field. 
We shall denote by $\nu:I_1\times I_2\to \R^3_*$ the co-normal vector field of the immersion (see \cite{Nomizu94},\cite{LiSimon93}).

\section{Improper affine spheres associated with a pair of planar curves}

\subsection{Generalized area distance}

Consider two plane curves $\alpha(s)=(\alpha_1(s),\alpha_2(s))$ and $\beta(t)=(\beta_1(t),\beta_2(t))$, where $\alpha_i(s)$ and $\beta_i(t)$ are $C^{\infty}$-functions.
Define
\begin{eqnarray*}
\x(s,t)&=&\frac{1}{2}\left(\alpha(s)+\beta(t)\right)=\frac{1}{2}\left(\alpha_1(s)+\beta_1(t),\alpha_2(s)+\beta_2(t)\right)\\
\n(s,t)&=&\frac{1}{2}\left(\beta(t)-\alpha(s)\right)^{\perp}=\frac{1}{2}\left(\alpha_2(s)-\beta_2(t),\beta_1(t)-\alpha_1(s)\right),
\end{eqnarray*}
where the symbol $\perp$ means anticlockwise rotation of ninety degrees.  
Geometrically, the point $\x$ is the midpoint of the {\it chord} connecting $\alpha(s)$ with $\beta(t)$ and $\n$ is orthogonal to the same chord, with half of its length.

Define $g(s,t)$ by the relation $\nabla g=\n$, where the gradient is taken with respect to $\x$. 
\begin{lemma}
The function $g$ is well defined up to a constant.
\end{lemma}
\begin{proof}
Assuming the existence of $g$ we calculate
\begin{eqnarray*}
g_s=\frac{1}{4}\left( (\alpha_2(s)-\beta_2(t))\alpha_1'(s)+ (\beta_1(t)-\alpha_1(s))\alpha_2'(s)  \right)=\frac{1}{4}\left[\beta(t)-\alpha(s),\alpha'(s)\right] \\
g_t=\frac{1}{4}\left( (\alpha_2(s)-\beta_2(t))\beta_1'(t)+ (\beta_1(t)-\alpha_1(s))\beta_2'(t)  \right)=\frac{1}{4}\left[\beta(t)-\alpha(s),\beta'(t)\right].
\end{eqnarray*}

\noindent
For the existence of $g$, we must prove that $(g_s)_t=(g_t)_s$, where $g_s$ and $g_t$ are defined above. But 
\begin{equation*}
(g_s)_t=(g_t)_s=\frac{1}{4}\left[\beta'(t),\alpha'(s)\right],
\end{equation*}
thus proving the lemma.
\end{proof}

We now give a geometric interpretation of $g(s,t)$. Let $\F=\frac{1}{2} (-y,x)$ denote a vector field in the plane whose line integral along a closed contour gives the  area of the region bounded 
by it. We fix a chord $L_0$ connecting $\alpha(0)$ with $\beta(0)$ and denote the line integral of $\F$ along $L_0$ by $C_0$. Then the signed area $A(s,t)$ of the region bounded by $L_0$,
$\alpha([0,s])$, the chord connecting $\alpha(s)$ with $\beta(t)$ and $\beta([t,0])$ (see figure \ref{AreaChords})  is given by
$$
2A(s,t)=C_0+\int_0^t\left[\beta(t),\beta'(t)\right]dt+\int_0^s\left[\alpha'(s),\alpha(s)\right]ds+\left[\beta(t),\alpha(s)\right].
$$
Thus 
\begin{eqnarray*}
2A_s&=&\left[\alpha'(s),\alpha(s)-\beta(t)\right]\\
2A_t&=&\left[\beta'(t),\alpha(s)-\beta(t)\right].
\end{eqnarray*}
We conclude that $g(s,t)=A(s,t)/2+C$, for some constant $C$. We call the map $(s,t)\to(\x(s,t),g(s,t))$ the {\it generalized area distance} of the pair of planar curves $(\alpha,\beta)$.

\begin{figure}[htb]
 \centering
 \includegraphics[width=0.45\linewidth]{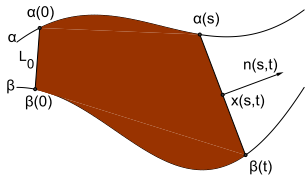}
 \caption{Generalized area distance of the pair of curves $(\alpha,\beta)$.}
\label{AreaChords}
\end{figure}

\subsection{ Improper indefinite affine spheres}

Consider the generalized area distance map $\q:I_1\times I_2\to\R^3$,
\begin{equation*}
\q(s,t)=\left(  \x(s,t), g(s,t)\right).
\end{equation*}
Then 
\begin{eqnarray*}
\q_s&=&\frac{1}{2}\left(  \alpha'(s),  \frac{1}{2}\left[  \beta(t)-\alpha(s),\alpha'(s)  \right]         \right)\\
\q_t&=&\frac{1}{2}\left( \beta'(t), \frac{1}{2} \left[ \beta(t)-\alpha(s),\beta'(t)  \right]         \right).
\end{eqnarray*}
Since
\begin{equation*}
\q_{ss}= \frac{1}{2} \left(    \alpha''(s),  \frac{1}{2}\left[  \beta(t)-\alpha(s),\alpha'(s)  \right]       \right),
\end{equation*}
we observe that $\q_{ss}$ is a linear combination of $\q_s$ and $\q_t$. This holds also for $\q_{tt}$, and so we conclude that $L=N=0$, i.e., $(s,t)$ are asymptotic coordinates. 
Since 
\begin{equation*}
\q_{st}=\left( 0, 0, g_{st}\right),
\end{equation*}
we obtain $M=\left(\frac{1}{4} \left[ \alpha'(s),\beta'(t) \right]\right)^2$ and thus
$$
\Omega(s,t)=\frac{1}{4} \left[ \alpha'(s),\beta'(t) \right]=g_{st}.
$$
Now equation \eqref{AffineNormal} implies that $\xi=(0,0,1)$ and so $\q$ is an improper affine sphere.

We also write
\begin{eqnarray*}
\q_{ss}&=& \frac{1}{\Omega} \left( \left[\alpha''(s),\beta'(t)\right]  \q_s+\left[ \alpha'(s),\alpha''(s)  \right]\q_t       \right)\\
\q_{tt}&=&\frac{1}{\Omega} \left( - \left[\beta'(t),\beta''(t)    \right]\q_s+\left[\alpha'(s),\beta''(t)  \right]\q_t       \right),
\end{eqnarray*}
and so the cubic form is given by
\begin{equation}\label{cubic}
C(s,t)=a(s)^3ds^3-b(t)^3dt^3,
\end{equation}
where 
$$
a(s)=\left[\alpha'(s),\alpha''(s)\right]^{1/3},\ \ \ \ b(t)=\left[\beta'(t),\beta''(t)\right]^{1/3}.
$$
Thus the cubic form vanishes in the direction of the line
$a(s)ds=b(t)dt$. One can easily verify that $\nu=(-\n,1)$ is the co-normal vector field of this immersion. In figure \ref{ASTwoCurves} one can see an improper 
indefinite affine sphere obtained from a pair of planar curves.

\begin{figure}[htb]
 \centering
 \includegraphics[width=0.75\linewidth]{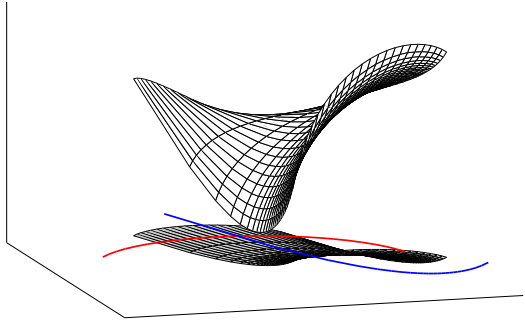}
 \caption{An improper indefinite affine sphere  obtained from two planar curves, together with its projection in the (x,y)-plane. }
\label{ASTwoCurves}
\end{figure}

\subsection{Generality of the construction}

In this section we show that any smooth improper indefinite affine sphere is the generalized area distance of  a pair of planar curves.

Consider a smooth improper affine sphere. We may assume that the constant affine normal is $\xi=(0,0,1)$ and write the affine immersion $q:I_1\times I_2\to\R^3$
in the form $q(s,t)=(\x(s,t),g(s,t))$, where $(s,t)$ are asymptotic coordinates of the immersion. Let $\n(s,t)=\nabla g(s,t)$ and define $\alpha(s,t)=\x(s,t)-\n(s,t)^{\perp}$, $\beta(s,t)=\x(s,t)+\n(s,t)^{\perp}$.

\begin{lemma}
$\alpha_t=\beta_s=0$.
\end{lemma}

In the proof of this lemma, we shall denote by $D^2g:\R^2\times\R^2\to\R$ the hessian of $g$, which is a symmetric bilinear form. By fixing the inner product of $\R^2$, $D^2g$ can also be seen as a self-adjoint linear
operator of $\R^2$.

\begin{proof}
Observe that since $(s,t)$ are asymptotic directions $D^2g(\x_s)=\lambda\x_s^{\perp}$ and $D^2g(\x_t)=\mu\x_t^{\perp}$. Thus $D^2g(\x_s,\x_t)=\lambda[\x_s,\x_t]$ and $D^2g(\x_t,\x_s)=\mu[\x_t,\x_s]$. Since $D^2g$ is symmetric, $\mu=-\lambda$. 

For any improper indefinite affine sphere with $\xi=(0,0,1)$ we have that $\det(D^2g)=-1$ (\cite{LiSimon93}).  Thus $D^2g(\x_s,\x_t)=[\x_t,\x_s]$. This implies that 
$\lambda=\pm 1$. We shall assume that $\lambda=1$, otherwise we change the roles of $s$ and $t$. In this case,
$(D^2g(\x_s))^{\perp}=-\x_s$ and $(D^2g(\x_t))^{\perp}=\x_t$. 
Thus 
\begin{eqnarray*}
\alpha_t&=&\x_t-D^2g(\x_t)^{\perp}=0\\
\beta_s&=&\x_s+D^2g(\x_s)^{\perp}=0,
\end{eqnarray*}
thus proving the lemma.
\end{proof}

\begin{Proposition}
Every smooth improper indefinite affine sphere is the generalized area distance of a pair of planar curves. 
\end{Proposition}

\begin{proof}
The above lemma implies that $\alpha=\alpha(s)$ and $\beta=\beta(t)$. Also, 
\begin{eqnarray*}
\x(s,t)&=&\frac{1}{2}\left(\alpha(s)+\beta(t)\right)\\
\nabla g(s,t)&=&\frac{1}{2}\left(\beta(t)-\alpha(s)\right),
\end{eqnarray*}
and thus $(\x(s,t),g(s,t))$ is the generalized distance of the pair $(\alpha(s),\beta(t))$. 
\end{proof}

\subsection{Properties of the representation}

One immediate property of the proposed representation formula for improper indefinite affine spheres is that it does not depend on the
parameterizations of the planar curves $\alpha$ and $\beta$. In fact, under a change of parameters of these curves, the map $q:I_1\times I_2\to \R^3$
changes but its image remains the same. 

Consider now an affine transformation of $\R^3$ that preserves the affine normal $\xi=(0,0,1)$. The corresponding transformation of the planar curves $\alpha$ and $\beta$ are the following: to a translation of $\R^3$ by a vector $(\w,w_3)$, the corresponding curves in the plane are translated by the vector $\w$. Similarly, for any affine transformation of $\R^3$ of the form $(\x,z)\to(A\x,z)$, where $A$ is an affine transformation of the plane, the corres\-ponding curves
are also transformed by $A$. The most interesting case occurs for affine transformations of $\R^3$ of the form $(\x,z)\to(\x, \a\cdot\x+z)$, where $\a$ is any vector of $\R^2$. In this case, the corresponding $\alpha$-curve
is translated by $-\a^{\perp}$ while the corresponding $\beta$-curve is translated by $\a^{\perp}$.

\section{Singularity set}



The singularity set $S$ of $\q$ consists of pairs $(s,t)$ for which $\Omega=0$, i.e., $[\alpha'(s),\beta'(t)]=0$ (see figure \ref{ASSingular}). 
Geometrically, the set $\x(S)$ consists of the midpoints of chords connecting $\alpha(s)$ and $\beta(t)$  with parallel tangents. In the theory of symmetry sets of planar curves, this set
is called area evolute, or  midpoint parallel tangent locus (\cite{Giblin08},\cite{Holtom00}). 
In case $\alpha(s)$ and $\beta(t)$ are parameterizations of the same planar curve without parallel tangents, the set $\x(S)$ is just the curve itself. 

\begin{figure}[htb]
 \centering
 \includegraphics[width=0.85\linewidth]{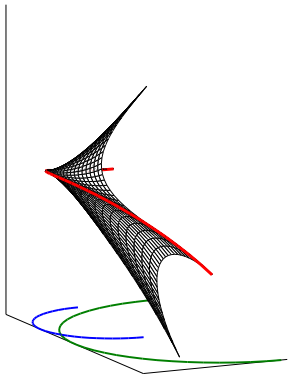}
 \caption{An affine sphere with singularities obtained from two arcs of circles. }
\label{ASSingular}
\end{figure}

\subsection{Indefinite IA-maps and non-degenerate singularities }

Folowing  \cite{Nakajo09}, the map $(s,t)\to\q(s,t)$  is called an {\it indefinite improper affine map}, or more concisely,  an {\it indefinite IA-map},
if the Lagrangian map $(s,t)\to(\x,\n)$  is an immersion. 
Since
\begin{eqnarray*}
(\x_s,\n_s)&=&\frac{1}{2}\left( \alpha'(s), - \alpha'(s)^{\perp}    \right)\\
(\x_t,\n_t)&=&\frac{1}{2}\left( \beta'(t), \beta'(t)^{\perp}       \right),
\end{eqnarray*}
we conclude that $(s,t)\to\q(s,t)$ is an indefinite IA-map if and only if 
\begin{equation}\label{condIAmaps}
\alpha'(s)\neq 0,\ s\in I_1\ \ \ {\text and}\ \ \ \beta'(t)\neq 0,\ t\in I_2. 
\end{equation}
From now on we shall always assume that the parameterized curves $\alpha(s)$ and $\beta(t)$ are regular, i.e., $\alpha'(s)\neq 0$, $s\in I_1$ and $\beta'(t)\neq 0$, $t\in I_2$.

A singularity is {\it non-degenerate} if $(\Omega_s,\Omega_t)\neq (0,0)$ (see \cite{Nakajo09}). We have that
\begin{eqnarray*}
\Omega_s&=& [\alpha''(s),\beta'(t)]\\
\Omega_t&=& [\alpha'(s),\beta''(t)].
\end{eqnarray*}
The condition $\Omega_t=0$ at a singular point implies that $\beta(t)$ is an inflection. Thus
a singularity is degenerate if and only if  $\alpha(s)$ and $\beta(t)$ are inflection points. 
The existence of pairs $(s,t)$ corresponding to inflection points of the curves and such that $\alpha'(s)$ is parallel to $\beta'(t)$ is  non-generic phenomena and, from now on, we shall
assume that it does not occur. 

To resume, we shall assume that \eqref{condIAmaps} holds and that there are not any pair of inflection points $(\alpha(s),\beta(t))$ with 
parallel tangents, which is equivalent to assume that the map $(s,t)\to\q(s,t)$ is an indefinite IA-map with non-degenerate singular set.
Under these assumptions, the singularity set $S$ in the $(s,t)$-plane is always smooth and the cubic form $C(s,t)$
of the immersion $(s,t)\to\q(s,t)$ defined by equation \eqref{cubic} is non-zero at any singular point.

\subsection{Classification of the singular points}

Parameterize the singular set $S$ by $r\to\gamma(r)=(s(r),t(r))$ and denote by $\eta(r)$ the null direction of $d\x$. In order to verify the 
regularity of $\x(S)$ at a point $\x(\gamma(r_0))=\x(s_0,t_0)$, we must check 
whether or not $\eta(r_0)$ is parallel to $\gamma'(r_0)$. Equivalently, we must verify whether or not 
$$
d\x\cdot(\Omega_t,-\Omega_s)= \Omega_t\alpha'(s_0)-\Omega_s\beta'(t_0)
$$
is the null vector. 


We have the following lemma:
\begin{lemma}\label{regularsingular}
The following statements are equivalent:
\begin{enumerate}
\item  The affine tangent vectors of $\alpha$ at $s_0$ and $\beta$ at $t_0$ are opposite, i.e., $b(t_0)\alpha'(s_0)+a(s_0)\beta'(t_0)=0$.
\item The euclidean curvatures $k_1$ of $\alpha$ at $s_0$ and $k_2$ of $\beta$ at $t_0$ are equal. 
\item The null direction $\eta$ of $d\x$ is tangent to $S$, i.e., $\Omega_t\alpha'(s_0)-\Omega_s\beta'(t_0)=0$.
\item The direction $(b(t_0),a(s_0))$ that vanishes the cubic form is tangent to $S$. 
\end{enumerate}
Any point $(s,t)\in S$ that does not satisfy these conditions will be called a {\it regular singular point}. 
\end{lemma}

\begin{proof}
We first observe that if $s_0$ corresponds to a inflection point of $\alpha$, then all three statements are false. This is a direct consequence of the fact that, at an inflection 
point of $\alpha$, $\Omega_s=0$. A similar fact holds for $\beta(t_0)$, and thus we may assume that neither $\alpha$ nor $\beta$ has inflection points. 
Thus we may assume $\alpha$ and $\beta$ are parameterized by affine arc length. In this case the direction that vanishes the cubic form is $(1,1)$
and the affine tangents are $\alpha'(s_0)$ and $\beta'(t_0)$. Under this simplification, the equivalence between (1) and (2) is an easy exercise.

\noindent
Writing $\alpha'(s_0)=-\lambda\beta'(t_0)$, we have $\eta=(1,\lambda)$ and so 
 $$
 \Omega_s=[\alpha''(s_0),\beta'(t_0)]=\frac{1}{\lambda}[\alpha'(s_0),\alpha''(s_0)]=\frac{1}{\lambda}. 
 $$
 Similarly, $\Omega_t=-\lambda$. Thus 
 $$
 \Omega_t\alpha'(s_0)-\Omega_s\beta'(t_0)=\left(\lambda^2-\frac{1}{\lambda}\right)\beta'(t_0).
 $$
We conclude that $\eta=(\Omega_t,-\Omega_s)$ if and only if $\lambda=1$. In other words, 
(1) and (3) are equivalent. Also, $\gamma'(r_0)$ is parallel to
$(-\Omega_t,\Omega_s)=(\lambda,\frac{1}{\lambda})$, and thus is also parallel to $(1,1)$ if and only if $\lambda=1$. We conclude that (4) is equivalent to (1).
\end{proof}

We shall use here proposition 1.3 of \cite{Kokubu05}, that we transcribe below for the reader's convenience:

\begin{Proposition}\label{Kokubu}
Let $\gamma(r_0)$ be a nondegenerate singularity of a front $q:I_1\times I_2\to\R^3$. Then
\begin{enumerate}
\item The germ of $q$ at $\gamma(r_0)$ is locally diffeomorphic to a cuspidal edge if and only if $\eta(r_0)$ is not parallel to $\gamma'(r_0)$.
\item The germ of $q$ at $\gamma(r_0)$ is locally diffeomorphic to a swallowtail if and only if $\eta(r_0)$ is parallel to $\gamma'(r_0)$ and
$A'(r_0)\neq 0$, where 
$$
A(r)=\det(\eta(r),\gamma'(r)).
$$ 
\end{enumerate}
\end{Proposition}

Following \cite{Nakajo09}, a singularity $(s_0,t_0)$ is called is called a {\it front} if  
the map $(s,t)\to (\q,\nu)$ is an immersion at $(s_0,t_0)$. One can easily verify that this is equivalent to 
the Lagrangian map $(s,t)\to(\x,\n)$ being an immersion at the same point. So every singularity of an indefinite IA-map
is a front singularity, and thus we are able to apply proposition \ref{Kokubu}. 

\begin{corollary}\label{cuspidal}
Assume that $(s_0,t_0)$ is a regular singular point. Then $\x(S)$ is smooth at $(s_0,t_0)$ and the germ of singularity $q(s_0,t_0)$ is 
diffeomorphic to a cuspidal edge.
\end{corollary}
\begin{proof}
At points where $\eta$ is transversal to $S$, the singularity set $\x(S)$ is smooth. In fact, one easily verifies that the tangent to $\x(S)$ at $\x(s,t)$ is parallel
to $\alpha'(s)$ and $\beta'(t)$. Now, using item (3) of the lemma \ref{regularsingular} and proposition \ref{Kokubu}(a), the germ of $q(s,t)$ at $\q(s_0,t_0)$ is diffeomorphic to a cuspidal edge.
\end{proof}

Consider a singularity $(s_0,t_0)$ associated with the parameter $r_0$ such that the conditions of lemma \ref{regularsingular} holds. We may assume that, 
close to $(s_0,t_0)$, $\alpha(s)$ and $\beta(t)$ are parameterized by affine arc-length. Define $\lambda(r)$ by the equation $\alpha'(r)+\lambda(r)\beta'(r)=0$. 
 
\begin{lemma}\label{singularsingular}
The following statements are equivalent:
\begin{enumerate}
\item $\lambda'(r_0)\neq 0$.
\item $k_1'(r_0)\neq k_2'(r_0)$.
\item $A'(r_0)\neq 0$.
\end{enumerate}
\end{lemma}
\begin{proof}
We have that 
$$
k_1(r)=\frac{1}{\lambda^3(r)}k_2(r)
$$ 
and thus 
$$
k_1'(r_0)=k_2'(r_0)-3\lambda'(r_0)k_2(r_0).
$$
Thus (1) and (2) are equivalent.
Since $\eta(r)=(1,\lambda(r))$ and $\gamma'(r)$ is parallel to $(\lambda(r),\frac{1}{\lambda(r)})$, we obtain that
$$
A(r)=c(r)\left(\frac{1}{\lambda(r)}-\lambda^2(r)\right),
$$
where $c(r_0)\neq 0$. So 
$$
A'(r_0)=-3c(r_0)\lambda'(r_0),
$$
and thus the equivalence between (1) and (3) is proved.
\end{proof}

\begin{corollary}\label{swallow}
Suppose that $(s_0,t_0)\in S$ is not regular and the conditions of lemma \ref{singularsingular} hold.
Then $\x(S)$ has an ordinary cusp at $(s_0,t_0)$ and the germ of $q(s,t)$ at $\q(s_0,t_0)$ is diffeomorphic to a swallowtail. 
\end{corollary}
\begin{proof}
From condition (2) of lemma \ref{singularsingular}, we can use proposition 2.4.9 of \cite{Holtom00} to conclude that $\x(S)$ has an ordinary cusp at $(s_0,t_0)$.
From condition (3) of lemma \ref{singularsingular}, we can use proposition \ref{Kokubu}(b) to conclude that the germ of $q(s,t)$ at $q(s_0,t_0)$ is diffeomorphic to a swallowtail. 
\end{proof}

\section{Other symmetry sets and their corresponding points in the affine sphere}

Besides the area evolute, we consider here two more symmetry sets associated to a pair of planar curves, the AASS and the AESS.

\subsection{Self-intersections of the affine sphere and the affine area symmetry set}

The {\it affine area symmetry set (AASS)} is the locus of midpoints of at least two chords $\overline{\alpha(s_1)\beta(t_1)}$ and $\overline{\alpha(s_2)\beta(t_2)}$ that determine 
the same area with respect to a fixed
chord. We denote by $A$ the set of pairs $(s,t)$ such that $\x(s,t)$ belongs to the AASS.  It is easy to see that a point $\q=\q(s_1,t_1)=\q(s_2,t_2)$ is a point of self-intersection of the
improper affine sphere if and only if $(s_1,t_1)$ and $(s_2,t_2)$ are corresponding points of $A$.

One property of the AASS is that the tangent line to $\x(A)$ is parallel to $\alpha(s_1)-\alpha(s_2)$ and 
$\beta(t_1)-\beta(t_2)$. Moreover, its endpoints belong to the area evolute, in fact are singular points of the area evolute (\cite{Giblin08},\cite{Holtom00}). 
We conclude that the tangent line to the self-intersection curve of the affine sphere is contained in a vertical plane parallel to $\alpha(s_1)-\alpha(s_2)$
and $\beta(t_1)-\beta(t_2)$. Also, the endpoints of the self-intersection curve  belong to $\q(S)$ but are not regular singular points.

\subsection{Symmetries of the affine sphere and the affine envelope symmetry set}

The {\it affine envelope symmetry set (AESS)} is the locus of centers of $3+3$ conics, i.e., conics that have third order contact
 with $\alpha$ at $\alpha(s)$ and with $\beta$ at $\beta(t)$. 
We shall denote by $E\subset I_1\times I_2$ the set of pairs 
$(s,t)$ such that there exists a $3+3$ conic at $(\alpha(s),\beta(t))$ and  by $\c(E)$ the corresponding centers. 
The line $l(s,t)$ that passes through $\x(s,t)$ and the intersection of the tangent lines to the curves $\alpha$ at $\alpha(s)$ and $\beta$ at $\beta(t)$ is called the {\it midline}. 
If $\alpha'(s)$ is parallel to $\beta'(t)$, the midline is just the line through $\x(s,t)$ parallel to these vectors. 
One can prove that the midline $l(s,t)$ is always tangent to $\c(E)$ at a point $\c(s,t)$ (see \cite{Giblin98}).



For $(s,t)\in E$, the affine tangents to the 3+3 conic coincide with the affine tangents to $\alpha$ and $\beta$ at $\alpha(s)$ and $\beta(t)$. 
Consider the planar affine reflection $T$ in the direction of the chord $\overline{\alpha(s)\beta(t)}$ and fixing the midline $l(s,t)$. Then $T$ takes the affine tangent at $s$ to the affine tangent
at $t$. We can lift this planar affine reflection $T$ to a spatial affine reflection ${\overline T}$ in the direction of the tangent line of the level set of $g$ and fixing the
 vertical plane defined by the direction of the midline. Then ${\overline T}$ takes the asymptotic directions $(\q_s,\q_t)$ 
 to $(\q_t,\q_s)$. Thus it preserves the Blaschke metric at the point $\q(s,t)\in\q(E)$. Also, since $T$ preserves the affine tangents at $\alpha(s)$ and $\beta(t)$, ${\overline T}$ preserves the cubic form
 at $\q(s,t)$. Thus we may say that $\q(E)$ is the set of points  of {\it local affine symmetry} of the affine sphere.

If a pair $(s,t)$ belongs to $E\cap S$, then $\x(s,t)$ is necessarily a singular point of $\x(S)$ (\cite{Giblin08},\cite{Holtom00}). 
We conclude that if a point $q(s,t)\in\q(E)\cap\q(S)$, then it is not a regular singular point.

\section{Examples}

\begin{example}\label{ExCusp1}
Let
\begin{eqnarray*}
\alpha(s)&=&\left(    s,   s^2+s^3-1 \right) \ \ \ -\epsilon\leq s\leq \epsilon\\
\beta(t)&=&\left(     t,  1-t^2-t^3 \right) \ \ \ -\epsilon\leq t\leq \epsilon
\end{eqnarray*}
Then
\begin{eqnarray*}
\x&=&\frac{1}{2}\left(s+t, s^2+s^3-t^2-t^3  \right)\\
\n&=&\frac{1}{2}\left( 2-s^2-s^3-t^2-t^3, s-t  \right)
\end{eqnarray*}
and straightforward calculations shows that
$$
g(s,t)=2s+2t+\frac{s^3}{3}+\frac{s^4}{2}+\frac{t^3}{3}+\frac{t^4}{2}-st^2-ts^2-st^3-ts^3.
$$
Since $4\Omega(s,t)=2s+3s^2+2t+3t^2$, a singular point $(s,t)$ must satisfy 
\begin{equation}\label{stcircle}
\left(3s+1\right)^2+\left(3t+1\right)^2=2. 
\end{equation}
Consider the parameterization 
\begin{equation*}
\left(s,t\right)=\left( -\frac{1}{3}+\frac{\sqrt{2}}{3}\cos(r), -\frac{1}{3}+\frac{\sqrt{2}}{3}\sin(r)\right)
\end{equation*}
of $S$. Since $a(s)=(2+6s)^{1/3}$ and $b(t)=-(2+6t)^{1/3}$, we obtain
$$
\lambda(r)=\tan^{1/3}(r).
$$
If $r_0\neq\frac{\pi}{4}$, $\lambda(r_0)\neq 1$ and thus corollary \ref{cuspidal} implies that $\x(r_0)$ is a regular point of $\x(S)$ and $\q(r_0)$ is a cuspidal edge of $\q$. 
For $r_0=\frac{\pi}{4}$, $\lambda(r_0)=1$ and $\lambda'(r_0)\neq 0$. 
We conclude from corollary \ref{swallow} that $\x=(0,0)$ is an ordinary cusp of $\x(S)$ and the germ of $\q$ at $(s,t)=(0,0)$ is diffeomorphic to a swallowtail (see figure \ref{Cusp1}).
One can also verify without much difficulty that, close to $(s,t)=(0,0)$, the AESS is the positive $x$-axis while the AASS is the negative $x$-axis. Thus the self-intersection set of the improper affine sphere projects into the negative $x$-axis and the local symmetry set $q(E)$ projects into the positive $x$-axis.

\begin{figure}[htb]
\centering \fsep \subfigure[ An ordinary cusp of $\x(S)$ at $(0,0)$.] {
\includegraphics[width=.25
\linewidth,clip =false]{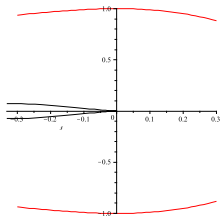}} \fsep\subfigure[
A swallowtail of $\q$ corresponding to the same singularity. ] {
\includegraphics[width=.40\linewidth,clip
=false]{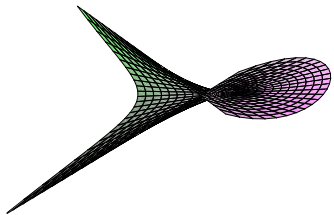}}\fsep
\caption{Singularities of example \ref{ExCusp1}.}
\label{Cusp1}
\end{figure}

\end{example}

\begin{example}\label{ExCusp2}
This example was considered in \cite{Nakajo09}. Let
\begin{eqnarray*}
\x(u,v)&=&\left(  u^2+v^2-u^3-3uv^2,     2uv+3u^2v+v^3\right)\\
\n(u,v)&=&\left( u^2+v^2+u^3+3uv^2,   -2uv+3u^2v+v^3\right).
\end{eqnarray*}
Changing to asymptotic coordinates $s=u+v, t=u-v$ we obtain
\begin{eqnarray*}
\x+\n^{\perp}&=&\alpha(s)=\left(    s^2 - s^3,   s^2+s^3 \right) \\
\x-\n^{\perp}&=&\beta(t)=\left(     t^2-t^3,  -t^2-t^3 \right).
\end{eqnarray*}
Then straightforward calculations shows that
\begin{eqnarray*}
\x&=&\frac{1}{2}\left(s^2-s^3+t^2-t^3,s^2+s^3-t^2-t^3\right)\\
g&=&  \frac{1}{2}\left(  s^2t^2-s^3t^3-\frac{s^5+t^5}{5}        \right) .
\end{eqnarray*}
Also $2\Omega(s,t)=-st(4-9st)$ and so the singularity set is given by $st=0$ and $9st=4$. 
The singular set corresponding to $9st=4$ can be parameterized by $\gamma_1(r)=\frac{2}{3}( r,r^{-1}), r>0$ and $\gamma_2(r)=\frac{2}{3}( r, r^{-1}), r<0$. Since 
$a(s)=12^{1/3}s^{2/3}$ and $b(t)=-12^{1/3}t^{2/3}$, straightforward calculations shows that $\lambda(r)=-r^{5/3}$, $r\neq 0$. 
So $\lambda(r_0)=1$ if and only if $r_0=-1$, and thus every point different from $s=t=-\frac{2}{3}$ corresponds to  a regular point of $\x(S)$ and a cuspidal edge
of $\q$. On the other hand, since $\lambda'(-1)\neq 0$, $s=t=-\frac{2}{3}$ corresponds to a cusp of $\x(S)$ and the germ of $q$ at this point is diffeomorphic to a swallowtail (see figure \ref{Cusp2}).

\begin{figure}[htb]
\centering \fsep \subfigure[ An ordinary cusp of $\x(S)$ together with the original curves $\alpha$ and $\beta$.] {
\includegraphics[width=.25
\linewidth,clip =false]{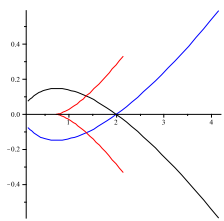}} \fsep\subfigure[
A swallowtail of $\q$ corresponding to the same singularity. ] {
\includegraphics[width=.40\linewidth,clip
=false]{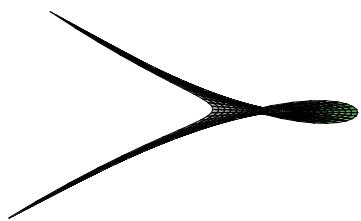}}\fsep
\caption{Singularities of example \ref{ExCusp2}.}
\label{Cusp2}
\end{figure}

\end{example}


\begin{thebibliography}{}

\bibitem{Buchin83}
S.Buchin, Affine Differential Geometry. Science Press, 1983.

\bibitem{Craizer08}
M.Craizer, M.A.da Silva, R.C.Teixeira: Area Distances of Convex Plane Curves and Improper Affine Spheres.
SIAM Journal on Mathematical Imaging, 1(3), p.209-227, 2008.




\bibitem{Giblin98}
P.J.Giblin, G.Sapiro, Affine Invariant Distances, Envelopes and Symmetry Sets, Geometriae Dedicata 71 (1998), p.237-261.

\bibitem{Giblin08} P.J.Giblin. Affinely invariant symmetry sets, Geometry and Topology of Caustics (Caustics 06), Banach Center Publications v.82 (2008), p.71-84.

\bibitem{Holtom00}
P.A.Holtom, Affine Invariant Symmetry Sets, University of Liverpool, Ph.D thesis, 2000.

\bibitem{Kokubu05} M.Kokubu, W.Rossman, K.Saji, M.Umehara, K.Yamada, Singularities of Flat Fronts in Hyperbolic Space. Pacific J.Math. 221 (2005), 303-351.

\bibitem{LiSimon93}
A.M.Li,  U.Simon, G.Zhao: Global Affine Differential Geometry of Hypersurfaces. De Gruyter Expositions in Mathematics, 11, 1993.

\bibitem{Martinez05}
A.Martinez, Improper Affine Maps, Mathematische Zeitschrift, 249 (2005), 755-766.

\bibitem{Nakajo09}
D.Nakajo, A Representation Formula for Indefinite Improper Affine Spheres, Results in Mathematics, 55 (2009), 139-159.

\bibitem{Niet04} M.Niethammer, S.Betelu, G.Sapiro, A.Tannenbaum, P.J.Giblin: Area-based Medial Axis of Planar Curves. 
International Journal of Computer Vision, 60(3), p.203-224, 2004.

\bibitem{Nomizu94}
K.Nomizu, T.Sasaki: Affine Differential Geometry. Cambridge University Press, 1994.



 
\end{thebibliography}
\end{document}